\documentclass[a4paper]{amsart}
\numberwithin{equation}{section}
\newtheorem{theorem}{Theorem}[section]
\newtheorem{corollary}{Corollary}[section]
\newtheorem{definition}{Definition}[section]
\newtheorem{lemma}{Lemma}[section]

\theoremstyle{remark}
\newtheorem{remark}{Remark}[section]
\usepackage{hyperref}
\usepackage{color}
\usepackage{graphicx}
\usepackage[margin=2cm]{geometry}
\usepackage{epstopdf}
\title[Further results for starlike functions]
 {Further results for starlike functions related with Booth lemniscate}
\begin{document}
\begin{abstract}
In this paper we investigate an interesting subclass $\mathcal{BS}(\alpha)$ ($0\leq \alpha<1$) of starlike functions in the unit disk $\Delta$. The class $\mathcal{BS}(\alpha)$ was introduced by Kargar et al. [R. Kargar, A. Ebadian and J. Sok\'{o}{\l}, {\it On Booth lemniscate and starlike functions}, Anal. Math. Phys. (2017) DOI: 10.1007/s13324-017-0187-3] which is strongly related to the Booth lemniscate. Some geometric properties of this class of analytic functions including, radius of starlikeness of order $\gamma$ ($0\leq\gamma<1$), the image of $f(\{z:|z|<r\})$ when $f\in \mathcal{BS}(\alpha)$, an special example  and estimate of bounds for ${\rm Re}\{f(z)/z\}$ are studied.
\end{abstract}


\maketitle
\section{Introduction}
Let $\mathcal{H}$ denote the class of analytic functions in the open unit disk $\Delta = \{z\in \mathbb{C} : |z| < 1\}$ on the complex plane $\mathbb{C}$. Also let $\mathcal{A}$ denote
the subclass of $\mathcal{H}$ including of functions normalized by $f(0)=f'(0)-1=0$. The subclass of $\mathcal{A}$ consists of all univalent functions $f(z)$ in
$\Delta$ is denoted by $\mathcal{S}$. We denote by $\mathfrak{B}$ the class of functions $w(z)$ analytic  in
$\Delta$ with $w(0) = 0$ and $|w(z)| < 1$, $(z \in \Delta)$. For two analytic and normalized functions $f$ and $g$, we say that $f$ is subordinate to $g$, written $f \prec g$ in $\Delta$, if there exists a function $w\in\mathfrak{B}$ such that $f (z) = g(w(z))$ for all
$z\in\Delta$.
In special case, if the function $g$ is univalent in $\Delta$, then
\begin{equation*}
    f (z)\prec g(z) \Leftrightarrow \left(f (0) = g(0)\quad {\rm and}\quad f (\Delta)\subset g(\Delta) \right).
\end{equation*}
It is easy to see that for any complex numbers $\lambda\neq0$ and $\mu$, we have:
\begin{equation}\label{lambdamu}
    f (z)\prec g(z) \Rightarrow \lambda f(z)+\mu\prec\lambda g(z)+\mu.
\end{equation}
The set of all functions $f\in \mathcal{A}$ that are starlike
univalent in $\Delta$ will be denoted by $\mathcal{S}^*$ and the set of all functions $f\in \mathcal{A}$ that are convex univalent in $\Delta$ will be denoted by $\mathcal{K}$. 
Robertson (see \cite{ROB}) introduced and studied the class $\mathcal{S}^*(\gamma)$ of starlike functions of order $\gamma\leq1$ as follows
\begin{equation*}
   \mathcal{S}^*(\gamma):=\left\{ f\in \mathcal{A}:\ \ {\rm Re} \left\{\frac{zf'(z)}{f(z)}\right\}> \gamma, \ z\in
    \Delta\right\}.
\end{equation*}
We note that if $\gamma\in[0,1)$, then a function in $\mathcal{S}^*(\gamma)$ is univalent. Also we say that $f\in\mathcal{K}(\gamma)$ (the class of convex functions of order $\gamma$) if and only if $zf'(z)\in \mathcal{S}^*(\gamma)$. In particular we put
$\mathcal{S}^*(0)\equiv \mathcal{S}^*$ and $\mathcal{K}(0)\equiv \mathcal{K}$.

Recently, Kargar et al. \cite{KarEba} introduced and studied a class functions related to the Booth lemniscate as follows.
\begin{definition}\label{defBS}{\rm(}see \cite{KarEba}{\rm)}
The function $f\in\mathcal{A}$ belongs to the class
$\mathcal{BS}(\alpha)$, $0\leq \alpha<1$, if it satisfies the
condition
  \begin{equation}\label{defsob}
     \left(\frac{zf'(z)}{f(z)}-1\right)\prec \frac{z}{1-\alpha z^2}\qquad (z\in\Delta).
  \end{equation}
\end{definition}
Recall that \cite{psok}, a one-parameter family of functions given by
\begin{equation}\label{Falpha}
  F_{\alpha}(z):=\frac{z}{1-\alpha z^2}=\sum_{n=1}^{\infty}\alpha^{n-1}z^{2n-1}\qquad (z\in\Delta,~0\leq \alpha\leq1).
\end{equation}
are starlike univalent when $0\leq \alpha\leq 1$ and are convex for $0 \le \alpha \le 3-2\sqrt{2}\approx 0.1715$. We have also $F_{\alpha}(\Delta)=D(\alpha)$, where
\begin{equation}\label{D(alpha)}
  D(\alpha)=\left\{x+iy\in\mathbb{C}: ~ \left(x^2+y^2\right)^2-\frac{x^2}{(1-\alpha)^2}-\frac{y^2}{(1+\alpha)^2}<0, (0\leq \alpha<1)\right\}
\end{equation}
and
\begin{equation}\label{D(1)}
  D(1)=\left\{x+iy\in\mathbb{C}: ~ \left(\forall t\in (-\infty,-i/2]\cup[i/2,\infty)\right)[x+iy\neq it]\right\}.
\end{equation}

It is clear that the curve
\begin{equation*}
  \left(x^2+y^2\right)^2-\frac{x^2}{(1-\alpha)^2}-\frac{y^2}{(1+\alpha)^2}=0\qquad (x, y)\neq(0, 0),
\end{equation*}
is the Booth lemniscate of elliptic type (see Fig. 1, for $\alpha=1/3$). For more details, see \cite{KarEba}.
\begin{figure}[htp]
 \centering
    \includegraphics[width=7cm]{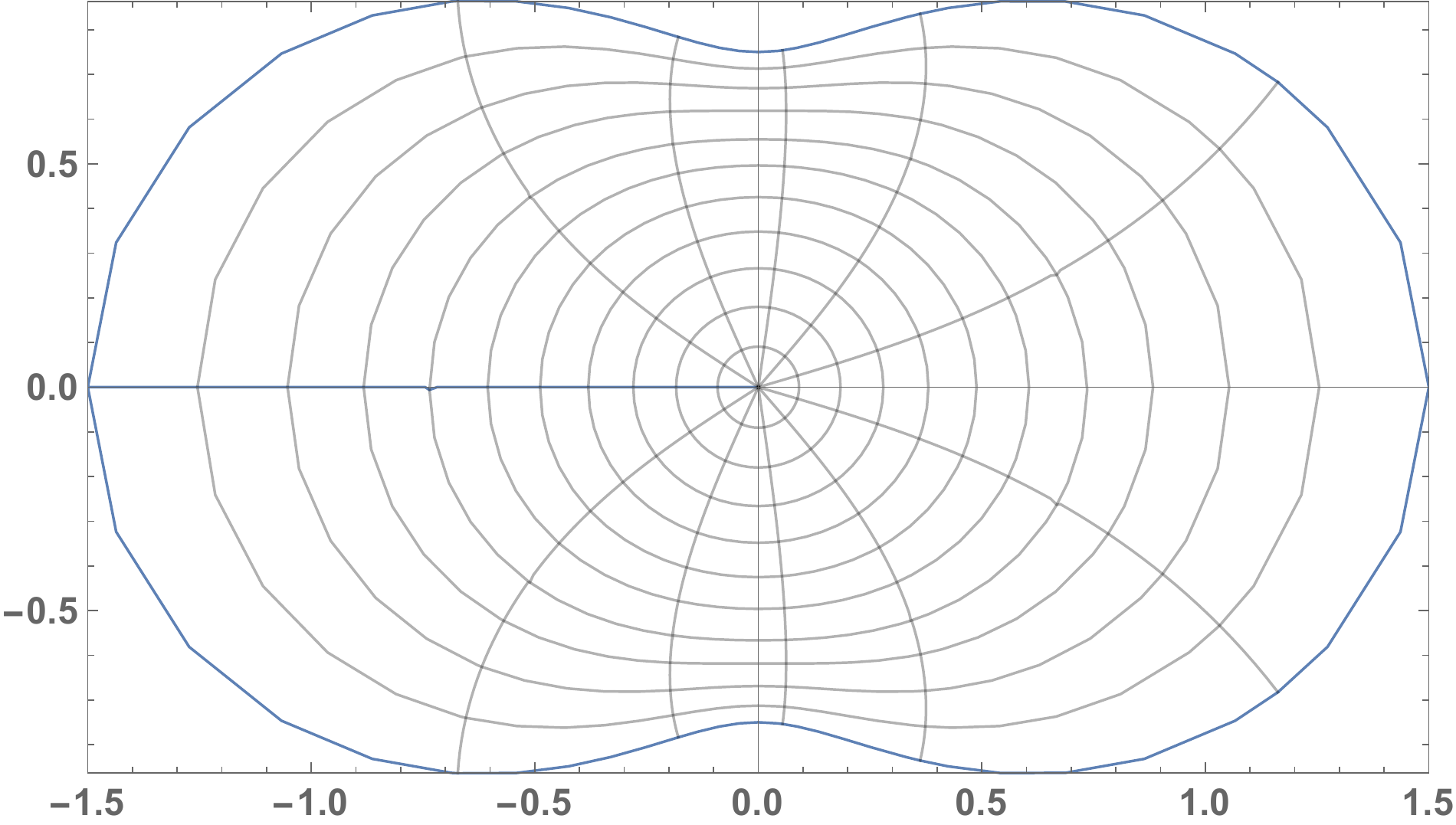}\\
  \caption{$\left(x^2+y^2\right)^2-9x^2/4-9 y^2/16=0$}\label{Fig:1}
\end{figure}
\begin{lemma}\label{LemKarEba}{\rm(}see \cite{KarEba}{\rm)}
Let $F_{\alpha}(z)$ be given by \eqref{Falpha}. Then for $0\leq
\alpha<1$, we have
  \begin{equation}\label{ReFalpha}
    \frac{1}{\alpha-1}< {\rm Re}\left\{F_{\alpha}(z)\right\}<\frac{1}{1-\alpha}\qquad (z\in\Delta).
  \end{equation}
\end{lemma}
Therefore by definition of subordination and by the Lemma \ref{LemKarEba}, $f\in\mathcal{A}$ belongs to the class
$\mathcal{BS}(\alpha)$, if it satisfies the condition
\begin{equation}\label{starlike}
    \frac{\alpha}{\alpha-1}< {\rm Re}\left\{\frac{zf'(z)}{f(z)}\right\}<\frac{2-\alpha}{1-\alpha}\qquad (z\in\Delta).
\end{equation}
The following lemma will be useful.
\begin{lemma}\label{Rushlemma}
  {\rm(}see \cite{RushStan}{\rm)} Let $F,G\in \mathcal{H}$ be any convex univalent functions in $\Delta$. If $f\prec F$ and $g\prec G$, then
  \begin{equation*}
    f*g\prec F*G\quad {in\ \ \Delta}.
  \end{equation*}
\end{lemma}

In this work, some geometric properties of the class $\mathcal{BS}(\alpha)$ are investigated.
\setcounter{lemma}{0}
\setcounter{theorem}{0}
\setcounter{equation}{0}
\section{\large Main results}
\noindent
We start with the following lemma that gives the structural formula for the function of the considered class.

\begin{lemma}\label{struct}
The function $f\in\mathcal{A}$ belongs to the class
$\mathcal{BS}(\alpha)$, $0\leq \alpha<1$, if and only if there exists an analytic function $q,$ $q(0)=0$ and $q\prec {F_{\alpha}}$ such that
  \begin{equation}\label{strform}
    f(z) = z \exp \left( \int_0^z \frac{q(t)}{t}dt \right).
  \end{equation}
\end{lemma}
The proof is easy.
Putting $q={F_{\alpha}}$ in Lemma \ref{struct} we obtain the function
 \begin{equation}\label{fwave}
    \tilde{f}(z) = z\left(\frac{1+z\sqrt{\alpha}}{1-z\sqrt{\alpha}} \right)^{\frac{1}{2\sqrt{\alpha}}},
  \end{equation}
which is extremal function  for several problems in the class $\mathcal{BS}(\alpha)$. Moreover, we consider
 \begin{equation}\label{F}
    F(z):=\frac{\tilde{f}(z)}{z} =\left(\frac{1+z\sqrt{\alpha}}{1-z\sqrt{\alpha}} \right)^{\frac{1}{2\sqrt{\alpha}}}
    =1+z+\frac{1}{2}z^2+\frac{1}{3}\left(\alpha+\frac{1}{2}\right)z^3\cdots.
  \end{equation}

From \eqref{starlike} we conclude that $f \in \mathcal{BS}(\alpha)$ is starlike of order $\frac{\alpha}{\alpha-1} < 0,$ hence $f$ may not be univalent in $\Delta$. It may therefore be interesting to consider a problem to find the radius of starlikeness of order $\gamma$, $\gamma \in [0,1)$ (hence univalence) of the class $\mathcal{BS}(\alpha)$, i.e. the largest radius $r_s(\alpha,\gamma)$ such that each function $f \in \mathcal{BS}(\alpha)$ is starlike of order $\gamma$ in the disc $|z|< r_s(\alpha,\gamma)$. For this purpose we recall the following property of the class $\mathfrak{B}.$

\begin{lemma}{\rm(}Schwarz lemma{\rm)} \label{Schw}{\rm(}see \cite{Dur}{\rm)}
Let $w$ be analytic in the unit disc $\Delta$, with $w(0)=0$ and $|w(z)|<1$ in $\Delta.$ Then $|w'(0)|\le 1$ and $|w(z)|\le |z|$ in $\Delta.$ Strict inequality holds in both estimates unless $w$ is a rotation of the disc: $w(z)=e^{i\theta}z.$
\end{lemma}

\begin{theorem}\label{rs}
Let $\alpha \in  (0,1)$  and $\gamma \in [0,1)$ be given numbers. If $f\in\mathcal{BS}(\alpha)$, then $f$ is starlike of order $\gamma$ in the disc $|z|< r_s(\alpha,\gamma)=\frac{\sqrt{1+4\alpha(1-\gamma)}-1}{2\alpha(1-\gamma)}.$ The result is sharp.
\end{theorem}

\begin{proof}
Let $f\in\mathcal{BS}(\alpha),$  $\alpha \in  (0,1).$ Then through \eqref{defsob} we have  $\left(\frac{zf'(z)}{f(z)}-1\right)\prec \frac{z}{1-\alpha z^2}$ so there exists $w\in\mathfrak{B}$ such that
\begin{equation*}
{\rm Re} \left\{\frac{zf'(z)}{f(z)}\right\}={\rm Re}\left\{\frac{1+w(z)-\alpha w^2(z)}{1-\alpha w^2(z)}\right\}
\end{equation*}
for all $z\in\Delta$. Applying the Schwarz lemma we obtain
 \begin{align*}
{\rm Re} \left\{\frac{zf'(z)}{f(z)}\right\}&={\rm Re}\left\{ 1+\frac{w(z)}{1-\alpha w^2(z)}\right\}=1+{\rm Re}\left\{\frac{w(z)}{1-\alpha w^2(z)}\right\}\\
&\ge 1-\frac{|w(z)|}{1-\alpha|w(z)|^2}\ge 1-\frac{|z|}{1-\alpha|z|^2}=1-\frac{r}{1-\alpha r^2},
\end{align*}
where $r=|z|<1$. Let us consider a function $h(r)=1-\frac{r}{1-\alpha r^2},$ $r \in (0,1).$ Note that $h'(r)=-\frac{1+\alpha r^2}{(1-\alpha r^2)^2}<0$ for all $r \in [0,1)$  hence $h$ is a strictly decreasing function and it decreases from $1$ to $\frac{\alpha}{\alpha -1}<0.$ Therefore the equation $h(r)=\gamma$ has for given $\alpha$ and $\gamma$ the smallest positive root $r_s(\alpha,\gamma)$ in $(0,1)$. Therefore $f$ is starlike of order $\gamma$ in $|z|<r\le r_s(\alpha,\gamma)$. Note that for the function $ \tilde{f}$ given in \eqref{fwave} we obtain
\begin{equation*}
   {\rm Re} \frac{z\tilde{f}'(z)}{\tilde{f}(z)}={\rm Re}\left\{ 1+\frac{z}{1-\alpha z^2}\right\}=:A(z)
  \end{equation*}
  and $A(-r_s(\alpha,\gamma))=\gamma$.
\end{proof}

Putting $\gamma=0$ in Theorem \ref{rs}, we obtain.
\begin{corollary}
Let $\alpha \in  (0,1)$. If $f\in\mathcal{BS}(\alpha)$ then $f$ is starlike univalent in the disc $|z|< r_s(\alpha)=\frac{\sqrt{1+4\alpha}-1}{2\alpha}.$ The result is sharp.
\end{corollary}

\begin{remark}
Note that $\lim_{\alpha \longrightarrow 0^+} r_s(\alpha)= \lim_{\alpha \longrightarrow 0^+} \frac{2}{\sqrt{1+4\alpha}+1}=1. $
Moreover, it is worth mentioning that $\lim_{\alpha \longrightarrow 1^-} r_s(\alpha)=\frac{\sqrt{5}-1}{2}= 0,618\dots$, i.e. this limit is a reciprocal of the  golden ratio $\frac{\sqrt{5}+1}{2}.$
\end{remark}

Now we consider the following question:\\ For a given number $r \in (0,1]$ find $\alpha(r)$ such that for each function $f\in\mathcal{BS}(\alpha(r))$ the image $f(\{z:|z|<r\})$ is a starlike domain.

\begin{theorem}\label{alpha_r}
Let $r \in  (0,1]$  be the given number. If  $0\le \alpha < \frac{1-r}{r^2}$, then each function $f\in\mathcal{BS}(\alpha)$ maps a disc $|z|<r$ onto a starlike domain.
\end{theorem}

\begin{proof}
After using the same argument as in the proof of Theorem \ref{rs} we conclude that  $f \in \mathcal{BS}(\alpha)$ satisfies the equality
\begin{equation*}
{\rm Re} \left\{\frac{zf'(z)}{f(z)}\right\}={\rm Re}\left\{\frac{1+w(z)-\alpha w^2(z)}{1-\alpha w^2(z)}\right\}
\end{equation*}
for all $z \in \Delta$ with some $w\in\mathfrak{B}$. Then we have by Schwarz's lemma that
\begin{equation*}
   {\rm Re} \left\{\frac{zf'(z)}{f(z)}\right\}\ge 1-\frac{|w(z)|}{1-\alpha|w(z)|^2}\ge 1-\frac{|z|}{1-\alpha|z|^2}.
  \end{equation*}
Consequently, for $|z|<r$, we obtain ${\rm Re} \left\{\frac{zf'(z)}{f(z)}\right\}>1-\frac{r}{1-\alpha r^2}.$ Let us note that a function $g(\alpha)=1-\frac{r}{1-\alpha r^2},$ $\alpha \in [0,1),$ has positive values  for $0\le \alpha < \frac{1-r}{r^2}$. Therefore the image of the disc  $|z|<r$ is a starlike domain.
\end{proof}

\begin{theorem}
Let $ n\ge 2$ be integer. If one of the following conditions holds

\begin{tabular}{llllllllllll}
$(i)$ & $\frac{1}{\alpha+n(1-\alpha)}< |c|<1,$\\
$(ii)$ & $n > \frac{3-\alpha}{1-\alpha}$\quad and \quad $\frac{1}{\alpha-2+n(1-\alpha)}< |c|<1,$\\
$(iii)$ & $n\ge\frac{2-\alpha}{1-\alpha}$\quad and \quad $|c|>1,$\\
$(iv)$ & $n<\frac{2-\alpha}{1-\alpha}$\quad and \quad $1<|c|<\frac{1}{2-\alpha+n(\alpha-1)},$\\
\end{tabular}\\
then the function $g_n(z)=z+cz^n$ does not belong to the class $\mathcal{BS}(\alpha)$.

\end{theorem}

\begin{proof}
Let us put $G(z)=\frac{zg'_n(z)}{g_n(z)}-1=\frac{1+cnz^{n-1}}{1+cz^{n-1}}-1.$ To prove our assertion it suffices to show that the  function $G$ is not subordinate to $F_{\alpha}$ or equivalently, because of the univalence of the dominant function $F_{\alpha}$, that the set $G(\Delta)$ is not included in $F_{\alpha}(\Delta)=D(\alpha).$ Upon performing simple calculation we find that the set $G(\Delta)$ is the disc with the diameter from the point $x_1=\frac{|c|(n-1)}{|c|-1}$ to the point $x_2=\frac{|c|(n-1)}{|c|+1}.$ The set $D(\alpha)$ is bounded by the curve
\begin{equation*}
  \left(x^2+y^2 \right)^2-\frac{x^2}{(1-\alpha)^2}-\frac{y^2}{(1+\alpha)^2}=0,\quad (x,y)\neq (0,0).
\end{equation*}
We have $\min_{|z|=1}{\rm Re}\{F_{\alpha}(z)\}=F_{\alpha}(-1)=\frac{1}{\alpha-1}$ and $\max_{|z|=1}{\rm Re}\{F_{\alpha}(z)\}=F_{\alpha}(1)=\frac{1}{1-\alpha}.$
If one of the conditions $(i)-(iv)$ is satisfied then $\min\{x_1,x_2\}<\frac{1}{\alpha-1}$ or $\max\{x_1,x_2\}>\frac{1}{1-\alpha},$ and then $G(\Delta)$ is not included in $D(\alpha).$ The proof of theorem is completed.
\end{proof}
Recently, one of the interesting problems for mathematician is to find bounds for ${\rm Re}\{f(z)/z\}$ (see \cite{kessiberian, sim2013}).
In the sequel, we obtain lower and upper bounds for ${\rm Re}\{f(z)/z\}$. We first get the following result for the function $F(z)$ given by \eqref{F}.

\begin{theorem}\label{F(z)-1}
  The function $F(z)-1$ is convex univalent in $\Delta$.
\end{theorem}
\begin{proof}
  Let us define
\begin{equation}\label{P(z)}
  p(z):=F(z)-1=\left(\frac{1+z\sqrt{\alpha}}{1-z\sqrt{\alpha}} \right)^{\frac{1}{2\sqrt{\alpha}}}-1=z+\frac{1}{2}z^2+\frac{1}{3}\left(\alpha+\frac{1}{2}\right)z^3\cdots.
\end{equation}
Then we see that $p(z)\in \mathcal{A}$. A simple calculation gives us
\begin{equation}\label{convex}
  1+\frac{zp''(z)}{p'(z)}=1+\left(\frac{1}{2\sqrt{\alpha}}-1\right)\left(\frac{2\sqrt{\alpha}z}{1-\alpha z^2}\right)+\frac{2\sqrt{\alpha}z}{1-\sqrt{\alpha}z}.
\end{equation}
It is sufficient to show that \eqref{convex} has positive real part in the unit disc. From Lemma \ref{LemKarEba} we obtain
\begin{eqnarray*}
  {\rm Re}\left\{1+\frac{zp''(z)}{p'(z)}\right\} &=& {\rm Re}\left\{1+\left(\frac{1}{2\sqrt{\alpha}}-1\right)\left(\frac{2\sqrt{\alpha}z}{1-\alpha z^2}\right)+\frac{2\sqrt{\alpha}z}{1-\sqrt{\alpha}z}\right\}\\
  &=&1+2\sqrt{\alpha}\left(\frac{1}{2\sqrt{\alpha}}-1\right){\rm Re}\left\{\frac{z}{1-\alpha z^2}\right\}+2\sqrt{\alpha}{\rm Re}\left\{\frac{z}{1-\sqrt{\alpha}z}\right\}\\
   &>& 1+\left(1-2\sqrt{\alpha}\right)\left(\frac{1}{\alpha-1}\right)
   -\frac{2\sqrt{\alpha}}{1+\sqrt{\alpha}}=:K(\alpha)\qquad (0\leq \alpha<1).
\end{eqnarray*}
It is easily seen that $K'(\alpha)=\frac{1}{(\alpha-1)^2}>0$. Thus $ K(\alpha)\geq K(0)=0$, and hence $F(z)-1$ is convex univalent function.
\end{proof}

In the proof of the next theorem we will use the following result concerning the convexity of the boundary of $D(\alpha)$.
\begin{lemma}\label{Dconv}{\rm(}see \cite{psok}{\rm)}
 Suppose that $F_\alpha$ is given by \eqref{Falpha}. If $0 \le \alpha \le 3-2\sqrt{2}\approx 0.1715$, then the curve $F_\alpha(e^{i\varphi})$, $\varphi \in [0,2\pi)$, is convex.  If $\alpha \in (3-2\sqrt{2},1)$, then the curve  $F_\alpha(e^{i\varphi})$, $\varphi \in [0,2\pi)$, is concave. Moreover, in both cases this curve is symmetric with respect to both axes.
\end{lemma}

\begin{theorem}\label{f/z sub}
  If a function $f$ belongs to the class $ \mathcal{BS}(\alpha)$, $0 \le \alpha \le 3-2\sqrt{2}$, then
  \begin{equation}\label{sub}
   \frac{f(z)}{z}\prec F(z)\qquad (z\in\Delta),
  \end{equation}
  where $F(z)$ is given by \eqref{F}.
\end{theorem}
\begin{proof}
  Let $0 \le \alpha \le 3-2\sqrt{2}$ and let $f$ be in the class $ \mathcal{BS}(\alpha)$. Then we have
  \begin{equation}\label{phi}
    \phi(z):=\frac{zf'(z)}{f(z)}-1\prec F_\alpha(z)\qquad (z\in \Delta),
  \end{equation}
  where $F_\alpha$ is given by \eqref{Falpha}. It is well known that the normalized function
  \begin{equation*}
    l(z)=\log \frac{1}{1-z}=\sum_{n=1}^{\infty}\frac{z^n}{n}\qquad (z\in \Delta),
  \end{equation*}
  belongs to the class $\mathcal{K}$ and for $f\in\mathcal{A}$ we get
  \begin{equation}\label{phil}
    \phi(z)*l(z)=\int_{0}^{z}\frac{\phi(t)}{t}{\rm d}t\quad {\rm and}\quad F_\alpha(z)*l(z)=\int_{0}^{z}\frac{F_\alpha(t)}{t}{\rm d}t.
  \end{equation}
  By Lemma \ref{Dconv} we deduce that the function $F_\alpha$ is convex. Thus applying Lemma \ref{Rushlemma} in \eqref{phi} we obtain
  \begin{equation}\label{philFl}
    \phi(z)*l(z)\prec F_\alpha(z)*l(z)\qquad (z\in \Delta).
  \end{equation}
  Now from \eqref{phil} and \eqref{philFl}, we can obtain
  \begin{equation*}
    \int_{0}^{z}\frac{\phi(t)}{t}{\rm d}t\prec \int_{0}^{z}\frac{F_\alpha(t)}{t}{\rm d}t\qquad (z\in \Delta).
  \end{equation*}
  Thus
\begin{equation*}
  \frac{f(z)}{z}= \exp\int_{0}^{z}\frac{\phi(t)}{t}{\rm d}t\prec \int_{0}^{z}\frac{F_\alpha(t)}{t}{\rm d}t=\frac{\tilde{f}(z)}{z}.
\end{equation*}
This completes the proof of theorem.
\end{proof}
Here by combining Theorem \ref{F(z)-1}, Theorem \ref{f/z sub} and \eqref{lambdamu}, we get:

\begin{theorem}
  Let $f\in \mathcal{BS}(\alpha)$, $0 \le \alpha \le 3-2\sqrt{2}$ and $|z|=r<1$. Then
  \begin{equation}\label{Re f/z}
    \left(\frac{1-r\sqrt{\alpha}}{1+\sqrt{\alpha}} \right)^{\frac{1}{2\sqrt{\alpha}}}\leq {\rm Re}\left(\frac{f(z)}{z}\right)\leq\left(\frac{1+r\sqrt{\alpha}}{1-r\sqrt{\alpha}} \right)^{\frac{1}{2\sqrt{\alpha}}}\qquad(z\in\Delta).
  \end{equation}
  The result is sharp.
\end{theorem}
\begin{proof}
  By the subordination principle, we have:
  \begin{equation*}
    f(z)\prec g(z)\Rightarrow f(|z|<r)\subset g(|z|<r)\qquad (0\leq r<1).
  \end{equation*}
  From Theorem \ref{F(z)-1}, since $F(z)-1$ is convex univalent in $\Delta$, and it is real for real $z$,
  thus it maps the disc $|z|=r<1$ onto a convex set symmetric which respect to the real axis laying between $F(-r)-1$ and $F(r)-1$.  Now the assertion is obtained from Theorem \ref{f/z sub}.
\end{proof}

\end{document}